\documentclass[11pt]{amsart}
\usepackage{pgf}
\usepackage{epstopdf}
\usepackage{amsthm,amsmath,amssymb,amsfonts,amscd,epsfig,mathrsfs}
\usepackage{graphicx,enumerate,csquotes,epigraph}
\usepackage[all]{xy}
\usepackage{booktabs,delarray,mathtools,float,faktor,xfrac}
\usepackage{textgreek,verbatim,subcaption,tabularx,longtable,multirow,spalign}
\usepackage{footnote}\makesavenoteenv{tabular}\makesavenoteenv{table}
\usepackage{tikz}
\usetikzlibrary{cd,positioning,arrows,shapes,calc,arrows.meta}

\DeclareMathOperator{\Ima}{Im}

\theoremstyle{plain}
\newtheorem{theorem}{Theorem}[section]
\newtheorem{corollary}[theorem]{Corollary}
\newtheorem{proposition}[theorem]{Proposition}
\newtheorem{lemma}[theorem]{Lemma}

\theoremstyle{definition}
\newtheorem{definition}[theorem]{Definition}

\theoremstyle{remark}
\newtheorem{remark}[theorem]{Remark}

\tikzset{
  pics/torus/.style n args={3}{
    code = {
      \providecolor{pgffillcolor}{rgb}{1,1,1}
      \begin{scope}[
          yscale=cos(#3),
          outer torus/.style = {draw,line width/.expanded={\the\dimexpr2\pgflinewidth+#2*2},line join=round},
          inner torus/.style = {draw=pgffillcolor,line width={#2*2}}
        ]
        \draw[outer torus] circle(#1);\draw[inner torus] circle(#1);
        \draw[outer torus] (180:#1) arc (180:360:#1);\draw[inner torus,line cap=round] (180:#1) arc (180:360:#1);
      \end{scope}
    }
  }
}

\begin{document}

\begin{abstract}
We define the Bohr-Sommerfeld quantization via $T$-modules for a $b$-symplectic toric manifold and show that it coincides with the formal geometric quantization of \cite{GuilleminMirandaWeitsman18}. In particular, we prove that its dimension is given by a signed count of the integral points in the moment polytope of the toric action on the manifold.
\end{abstract}

\author{Pau Mir}
\address{Pau Mir,
Laboratory of Geometry and Dynamical Systems, Universitat Polit\`{e}cnica de Catalunya, Avinguda del Doctor Mara\~{n}on 44-50, 08028, Barcelona }
\email {pau.mir.garcia@upc.edu}
\author{Eva Miranda}
\address{Eva Miranda,
Laboratory of Geometry and Dynamical Systems $\&$ Institut de Matem\`atiques de la UPC-BarcelonaTech (IMTech), Universitat Polit\`{e}cnica de Catalunya, Avinguda del Doctor Mara\~{n}on 44-50, 08028, Barcelona \\ CRM Centre de Recerca Matem\`{a}tica, Campus de Bellaterra
Edifici C, 08193 Bellaterra, Barcelona }
\email{eva.miranda@upc.edu}
\author{Jonathan Weitsman}
\address{Jonathan Weitsman, Department of Mathematics, Northeastern University, Boston, MA 02115}
\email{j.weitsman@neu.edu}

\thanks{P. Mir is funded in part by the Doctoral INPhINIT - RETAINING grant LCF/BQ/DR21/11880025 of “la Caixa” Foundation}
\thanks{P. Mir and E. Miranda are partially supported by the AEI grant PID2019-103849GB-I00 of MCIN/ AEI /10.13039/501100011033}
\thanks{E. Miranda is supported by the Catalan Institution for Research and Advanced Studies via an ICREA Academia Prizes 2016 and 2021 and by the Spanish State Research Agency, through the Severo Ochoa and Mar\'{\i}a de Maeztu Program for Centers and Units of Excellence in R\&D (project CEX2020-001084-M)}
\thanks{J. Weitsman was supported in part by a Simons collaboration grant}

\title{Bohr-Sommerfeld quantization of $b$-symplectic toric manifolds}

\maketitle

\epigraph{Dedicated to Victor Guillemin}

\section{Introduction}

Singular symplectic manifolds appear in the investigation of geometrical and dynamical facets of non-compact manifolds as their natural compactifications. This is exemplified in the work undertaken by \cite{KiesenhoferMirandaScott16} and \cite{MirandaOms} on the restricted three-body problem. They also appear in the realm of quantization, where new procedures are required to extend the classic quantization scheme to the compactification of symplectic manifolds.

For compact manifolds coming from a physical system, one of the minimal requirements for a quantization model is that it is finite-dimensional, something that has long been sought after. This quantization challenge was tackled, for instance, by Guillemin, Miranda and Weitsman using the formal geometric quantization of \cite{Weitsman01} and \cite{Paradan09}, and the surprising result they proved in \cite{GuilleminMirandaWeitsman18} is that the formal geometric quantization of a $b$-symplectic manifold a is indeed a finite-dimensional vector space. This raised the natural question of whether there is a true geometric quantization of such a space. An answer was given in the affirmative in \cite{BravermanLoizidesSong} and \cite{LinLoizidesSjamaarSong21}, where virtual modules agreeing with the formal geometric quantization of \cite{GuilleminMirandaWeitsman18} were constructed analytically using index theory.

The purpose of this paper is to revisit the question in the context of Bohr-Sommerfeld quantization, as it was done in \cite{GuilleminSternberg83}, but restricting ourselves to the case of symplectic and $b$-symplectic toric manifolds.

We start revisiting a result of Guillemin and Sternberg \cite{GuilleminSternberg83}, which identifies the Bohr-Sommerfeld leaves of a symplectic manifold endowed with an integrable system with the integer points in the image of its moment map. This allows to read the geometric quantization of a symplectic toric manifold from its Delzant polytope. Then, we prove that the same can be done for $b$-symplectic toric manifolds.

Next, we prove our main result, Theorem \ref{thm:b-GQcoincidesFGQ}. It states that, for any integral $b$-symplectic toric manifold, the Bohr-Sommerfeld quantization with sign agrees with the formal geometric quantization of \cite{GuilleminMirandaWeitsman18}. For this, we need to introduce the definition of Bohr-Sommerfeld quantization with sign, which we do via $T$-modules.

The article is organized as follows. In Section \ref{section:preliminaries} we recall the necessary preliminaries on $b$-symplectic manifolds, Bohr-Sommerfeld and formal geometric quantization. In Section \ref{section:BSleavesviaMM}, following the idea of \cite{GuilleminSternberg83}, we prove that the Bohr-Sommerfeld leaves of an integral $b$-symplectic toric manifold can be obtained from the image of the moment map of the toric action. In Section \ref{sec:BSquantwithsign} we introduce the Bohr-Sommerfeld quantization with sign via $T$-modules. In Section \ref{section:equivalencegqfgq} we prove the equivalence between Bohr-Sommerfeld quantization and formal geometric quantization both for integral symplectic and $b$-symplectic toric manifolds.

\section{Preliminaries. $b$-Symplectic manifolds, Bohr-Sommerfeld quantization and formal geometric quantization}
\label{section:preliminaries}

In this preliminaries section we review the definitions of $b$-symplectic manifolds, Bohr-Sommerfeld quantization and formal geometric quantization. We also recall the results that we use in the other sections.

\subsection{$b$-Symplectic manifolds}

$b$-Symplectic geometry is a generalization of symplectic geometry that furnishes manifolds with boundary with a Poisson structure which lowers rank at a singular hypersurface. It is possible to associate a tangent and a cotangent bundle to the class of $b$-symplectic manifolds and to expand there the classical symplectic tools. This singular model, presented in \cite{GuilleminMirandaPires11}, \cite{GuilleminMirandaPires14} and \cite{GuilleminMirandaPiresScott15}, reveals to be useful for many families of physical problems for which symplectic manifolds are not enough.

We proceed to briefly summon the necessary definitions in $b$-symplectic geometry and we refer to the aforementioned articles for a complete overview. Recall that a \textit{$b$-manifold} is a pair $(M,Z)$ where $Z$ is a hypersurface in a manifold $M$ and a \textit{$b$-map} is a map $f:(M_1,Z_1) \longrightarrow (M_2, Z_2)$ between $b$-manifolds with $f$ is transverse to $Z_2$ and $Z_1 = f^{-1}(Z_2)$.

\begin{definition}[$b$-vector field]
A vector field on a $b$-manifold $(M,Z)$ is called a \textit{$b$-vector field} if it is tangent to $Z$ at every point $p\in Z$.
\label{def:b-vectorfield}
\end{definition}

Let $(M^n,Z)$ be a $b$-manifold. If $x$ is a local defining function for $Z$ on an open set $U\subset M$ and $(x,y_1,\ldots,y_{n-1})$ is a chart on $U$, then the set of $b$-vector fields on $U$ is a free $C^\infty(M)$-module with basis
$$(x {\frac{\partial}{\partial x}}, {\frac{\partial}{\partial y_1}},\ldots, {\frac{\partial}{\partial y_n}}).$$

There exists a vector bundle associated to this module called the \textit{$b$-tangent bundle} and denoted by $^b TM$. The \textit{$b$-cotangent bundle} $^b T^*M$ of $M$ is defined to be the vector bundle dual to $^b TM$.

For each $k>0$, let $^b\Omega^k(M)$ denote the space of sections of the vector bundle $\Lambda^k(^b T^*M)$, called \textit{$b$-de Rham $k$-forms}. For any defining function $f$ of $Z$, every $b$-de Rham $k$-form can be written as
\begin{equation}
\omega=\alpha\wedge\frac{df}{f}+\beta, \text{ with } \alpha\in\Omega^{k-1}(M) \text{ and } \beta\in\Omega^k(M).
\end{equation}

This decomposition enables us to extend the exterior operator $d$ to $^b\Omega(M)$ by setting
$$d\omega=d\alpha\wedge\frac{df}{f}+d\beta.$$

The right hand side agrees with the usual exterior operator $d$ on $M\setminus Z$ and extends smoothly over $M$ as a section of $\Lambda^{k+1}(^b T^*M)$. The fact that $d^2=0$ allows us to define a complex of $b$-forms, the \textit{$b$-de Rham complex}. The cohomology associated to this complex is the \textit{$b$-cohomology} and it is denoted by \textit{$^b H^*(M)$}. The elements of $^b\Omega^0(M)$ are also called \textit{$b$-functions} and the following definition characterizes them.

\begin{definition}[$b$-function]
The set of \textit{$b$-functions} $^b C^\infty(M)$ consists of functions with values in $\mathbb{R} \cup \{\infty\}$ of the form $$c\,\textrm{log}\vert f\vert + g,$$ where $c \in \mathbb{R}$, $f$ is a defining function for $Z$ and $g$ is a smooth function on $M$. The differential operator $d$ is defined as:
$$d(c\,\textrm{log}\vert f\vert + g):= \frac{c \, df}{f} + d g \in\, ^b\Omega^1(M),$$ where $d g$ and $d f$ are the standard de Rham derivatives.
\label{def:b-function}
\end{definition}

A special class of closed $2$-forms of the complex of $b$-forms is the class of \textit{$b$-symplectic forms} as defined in \cite{GuilleminMirandaPires14}. It contains forms with singularities and can be introduced formally for $b$-symplectic manifolds, making it possible to extend the symplectic structure from $M\backslash Z$ to the whole manifold $M$.

\begin{definition}[$b$-symplectic manifold]
Let $(M^{2n},Z)$ be a $b$-manifold and $\omega\in\,^b\Omega^2(M)$ a closed $b$-form. We say that $\omega$ is \textit{$b$-symplectic} if $\omega_p$ is of maximal rank as an element of $\Lambda^2(\,^b T_p^* M)$ for all $p\in M$. The triple $(M,Z,\omega)$ is called a \textit{$b$-symplectic manifold}.
\label{def:bsymplecticmanifold}
\end{definition}

The Mazzeo-Melrose Theorem describes the relationship between $b$-cohomology and de Rham cohomology and makes it natural to talk of integrality of $b$-forms.

\begin{theorem}[Mazzeo-Melrose]
\label{thm:mazzeomelrose}
The $b$-cohomology groups of $(M^{2n},Z)$ satisfy
$$^b H^*(M)\cong H^*(M)\oplus H^{*-1}(Z).$$
\end{theorem}

\begin{remark}
The integrality of a $b$-form $\omega$ in the sense of \cite{GuilleminMirandaWeitsman18} implies the integrality of the form on $M\setminus Z$. Since we will work with a line bundle on $M$ whose Chern class is given by the projection of $[\omega]$ to $H^*(M)$, its restriction to $M\setminus Z$ has Chern class $[\omega_{M\setminus Z}]$.
\end{remark}

\subsection{Symplectic and $b$-symplectic toric manifolds}

A classical theorem on the classification of symplectic toric manifolds is Delzant's Theorem. It characterizes any such manifold via its Delzant polytope, the convex polytope corresponding to the entire image of the moment map of the toric action.

\begin{theorem}[Delzant, \cite{Delzant88}]
Let $(M^{2n},\omega,\mu)$ be a toric manifold. Then, there is a bijective correspondence between the following two sets, which is given by the image of the moment map $\mu$:
$$\begin{array}{ccc} \{\hbox{toric manifolds}\}&\longrightarrow
&\{\hbox{Delzant polytopes}\}
\\ (M^{2n},\omega,\mu) & \longrightarrow & \mu(M)
\end{array}$$
\label{thm:delzant}
\end{theorem}
In a symplectic toric manifold $(M^{2n},\omega,\mu)$, the singularities of the moment map $\mu$ can only be of elliptic type, in the sense of Williamson \cite{Williamson36}. In fact, the singular leaves of the toric foliation are in correspondence with the points in the boundary of Delzant's polytope, as it is detailed in the following observation.

\begin{remark}
Let $(M^{2n},\omega,\mu)$ be a symplectic toric manifold and $\Delta$ its Delzant's polytope. For $k=1,\dots,n$, the points in the intersection of $k\leq n$ facets of $\Delta$ correspond to the leaves of $M$ where $\mu$ has $k$ singular elliptic components. In particular, the vertices of $\Delta$ correspond to the fixed points of $\mu$. On the other hand, and in the appropriate coordinates, the elliptic singular components of the moment map at any singular leaf can be written as a sum of squares \cite{Eliasson90}.
\label{rem:ellipticpoints}
\end{remark}

The classification of $b$-symplectic toric surfaces was established by Guillemin, Miranda, Pires and Scott. Their result is that any $b$-symplectic surface is either a $b$-symplectic sphere or a $b$-symplectic torus.

\begin{theorem}[Guillemin-Miranda-Pires-Scott, \cite{GuilleminMirandaPiresScott15}]
A $b$-symplectic surface with a toric $S^1$-action is equivariantly $b$-symplectomorphic to either $(S^2,Z)$ or $(T^2,Z)$, where $Z$ is a collection of latitude circles (in the $T^2$ case, an even number of such circles), the action is the standard rotation, and the $b$-symplectic form is determined by the modular periods of the critical curves and the regularized Liouville volume.
\label{thm:b-surfaces}
\end{theorem}

The theorem above, in fact, boils down from the following result on the classification of higher-dimensional $b$-symplectic toric manifolds also carried out in \cite{GuilleminMirandaPiresScott15}.

\begin{proposition}[Guillemin-Miranda-Pires-Scott, Remark 38 in \cite{GuilleminMirandaPiresScott15}]
Every $b$-symplectic toric manifold is either the product of a $b$-symplectic $T^2$ with a classic symplectic toric manifold, or it can be obtained from the product of a $b$-symplectic $S^2$ with a classic symplectic toric manifold by a sequence of symplectic cuts performed at the north and south “polar caps”, away from the critical hypersurface $Z$.
\label{prop:bdelzant}
\end{proposition}

The image of the moment map of a $b$-symplectic toric manifold is a \textit{$b$-Delzant polytope} and the classification of Proposition \ref{prop:bdelzant} is a consequence of the $b$-Delzant theorem (Theorem 35 of \cite{GuilleminMirandaPiresScott15}). The main skeleton of the proof builds up from the proposition below (Proposition 18 in \cite{GuilleminMirandaPiresScott15}).

\begin{proposition}
Let $(M^{2n},Z,\omega,\mu)$ be a $b$-symplectic toric manifold, $L$ a leaf of its symplectic foliation and $v_Z$ the modular weight of $Z$. Pick a lattice element $X \in \mathfrak{t}$ that represents a generator of $\mathfrak{t} / \mathfrak{t}_Z$ and pairs positively with $v_{Z}$. Then, there is a neighbourhood $L \times S^1 \times (-\varepsilon, \varepsilon) \cong U \subseteq M$ of $Z$ such that the $T^n$-action on $U\setminus Z$ has moment map
\begin{equation*}
\mu_{U\setminus Z}: L \times S^1 \times \left((-\varepsilon, \varepsilon)\setminus\{0\}\right) \rightarrow \mathfrak{t}^* \cong \mathfrak{t}_{Z}^* \times \mathbb{R},\text{\,\,}
(\ell, \rho, t) \mapsto (\mu_{L}(\ell), c\log|t|),
\label{eqn:localmm}
\end{equation*}
where $c$ is the modular period of $Z$, the map $\mu_{L}: L \rightarrow \mathfrak{t}_{Z}^*$ is a moment map for the $T_{Z}^{n-1}$-action on $L$, and the isomorphism $\mathfrak{t}^* \cong \mathfrak{t}_{Z}^* \times \mathbb{R}$ is induced by the splitting $\mathfrak{t} \cong \mathfrak{t}_Z \oplus \langle X \rangle $.
\label{prop:localmomentmap}
\end{proposition}

Since the moment map of a group action over a $b$-symplectic manifold is a $b$-function (see Definition \ref{def:b-function}), it can be unbounded due to the logarithm term. Hence, its image is neither convex (in the sense of classical analysis, for a more sophisticated notion of convexity confer \cite{GuilleminMirandaPiresScott17}) in general (see Figure \ref{fig:bspheremomentmap}) nor bounded. This is the main handicap when one tries to obtain a finite Bohr-Sommerfeld quantization of a $b$-symplectic toric manifold and the reason why we introduce the \textit{Bohr-Sommerfeld quantization with sign} in Section \ref{sec:BSquantwithsign}.

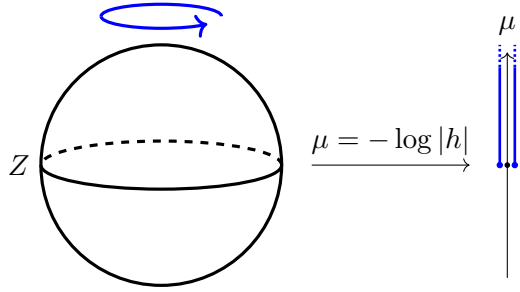
\begin{figure}[ht!]
\begin{center}
\begin{tikzpicture}
\pgfmathsetmacro{\rlinex}{6}
\pgfmathsetmacro{\baseptd}{8}
\pgfmathsetmacro{\rlineybottom}{2.75}
\pgfmathsetmacro{\rlineymid}{4.25}
\pgfmathsetmacro{\rlineytop}{5.75}
\def\R{1.6}
\pgfmathsetmacro{\circlex}{1.4}
\draw[dashed, very thick, color = black] (\circlex + \R, \rlineymid) arc (0:180:{\R} and {\R * .2});
\draw[very thick] (\circlex, \rlineymid) circle (\R);
\draw[very thick, color = black] (\circlex + \R, \rlineymid) arc (0:-180:{\R} and {\R * .2}) node[left] {$Z$};
\draw[->] (\circlex + 2, \rlineymid) -- node[above] {$\mu = -\log|h|$} (\rlinex - 0.5, \rlineymid);
\draw[->] (\rlinex, \rlineybottom) -- (\rlinex, \rlineytop) node[above,yshift=0.1cm] {$\mu$};
\draw[black, fill = black] (\rlinex, \rlineymid) circle(.3mm);
\draw[blue, fill = blue] (\rlinex - 0.1, \rlineymid) circle (1pt);
\draw[blue, fill = blue] (\rlinex + 0.1, \rlineymid) circle (1pt);
\draw[line width = 1pt, blue] (\rlinex - 0.1, \rlineytop-0.2) -- (\rlinex - 0.1, \rlineymid );
\draw[line width = 1pt, blue, densely dotted] (\rlinex - 0.1, \rlineytop-0.2) -- (\rlinex - 0.1, \rlineytop + 0.1);
\draw[line width = 1pt, blue] (\rlinex + 0.1, \rlineytop-0.2) -- (\rlinex + 0.1, \rlineymid);
\draw[line width = 1pt, blue, densely dotted] (\rlinex + 0.1, \rlineytop-0.2) -- (\rlinex + 0.1, \rlineytop + 0.1);
\draw [very thick, ->, blue] (\circlex, \rlineymid + 1.3*\R) ++(\R * -.5, -0.1) arc (180:320: {\R * .5} and {\R * .1});
\draw [very thick, blue] (\circlex, \rlineymid + 1.3*\R) ++(\R * -.5, -0.1) arc (180:0: {\R * .5} and {\R * .1});
\end{tikzpicture}
\end{center}
    \caption{The moment map of the rotation action over the canonical $b$-symplectic sphere is unbounded in any neighbourhood of $Z$.}
    \label{fig:bspheremomentmap}
\end{figure}

\subsection{Bohr-Sommerfeld quantization}
\label{subsec:geomquant}
Let us recall the Bohr-Sommerfeld quantization of a compact symplectic toric manifold using the classical definitions of Kostant \cite{Kostant70} and Guillemin-Sternberg \cite{GuilleminSternberg82}. Let $(M,\omega)$ be an integral symplectic manifold and let $\mathbb{L}$ be a complex line bundle with connection $\nabla$ whose curvature is $\omega$. Geometric quantization is a process which associates to the quadruple $(M,\omega,\mathbb{L},\nabla)$ a Hilbert space $Q(M)$ and to each $f\in C^{\infty}(M)$ a self-adjoint operator $Q(f)$ on that space.

Following \cite{Kostant70}, we define the quantization using the additional data given by a real polarization of $M$. That is, a foliation of $M$ by Lagrangian submanifolds. This foliation may be given by the fibres of a map $\pi:M\to B$ and, in this case, the quantization is given by sections $s\in \Gamma(\mathbb{L})$ satisfying
\begin{equation}
    \nabla_X s=0,
    \label{eq:flatsections}
\end{equation}
for any $X$ tangent to the fibres of $\pi$. If $(M,\omega,\mu)$ a toric manifold, a natural foliation is given by the fibres of the moment map $\mu: M\to \mathfrak{t}^*$.

When $M$ is compact, there are no smooth sections satisfying equation (\ref{eq:flatsections}) defined globally on all $M$. Instead, such \textit{leafwise constant sections} or \textit{flat sections} are concentrated on the fibres $\pi^{-1}(b)$ such that $b\in\Ima(\pi)$ and $\mathbb{L}|_{\pi^{-1}(b)}$ is a trivial bundle. The quantization space is defined as

\begin{equation}
    Q(M)=\bigoplus_{b\in B_{BS}}\mathbb{C}\langle s_b\rangle,
    \label{eq:BSquantclassical}
\end{equation}

where $B_{BS}$ is the Bohr-Sommerfeld set, namely, $$B_{BS}=\{b\in \Ima(\pi)\subset B \colon \mathbb{L}|_{\pi^{-1}(b)} \text{ is trivial}\},$$ and $s_b$ is the corresponding flat section of $\mathbb{L}|_{\pi^{-1}(b)}$.

This definition of quantization is based on the identification of the Bohr-Sommerfeld leaves and we call it \textit{Bohr-Sommerfeld quantization}. It coincides with the classical definition of geometric quantization of a symplectic manifold equipped with a completely integrable system that provides a real polarization. Some authors use sheaves to compute the Bohr-Sommerfeld leaves and obtain the same quantization (see \cite{Sniatycki77} or \cite{Hamilton08}).

\subsection{Formal geometric quantization}

We take the basic definitions of formal geometric quantization of Hamiltonian T-spaces from \cite{GuilleminMirandaWeitsman18}. For further reading on this quantization see \cite{Weitsman01}, \cite{HochsMathai16} and \cite{Paradan09}. 

Suppose $(M,\omega)$ is a compact symplectic manifold and let $(\mathbb{L}, \nabla)$ be a complex line bundle with connection of curvature $\omega$. By twisting the spin-$\mathbb{C}$ Dirac operator on $M$ by $\mathbb{L}$, we obtain an elliptic operator $\bar{\partial}_\mathbb{L}$. The geometric quantization $Q(M)$ of $M$ is defined by
$$Q(M) = {\rm ind}(\bar{\partial}_\mathbb{L}),$$
and it is a virtual vector space.

If $(M,\omega)$ is a compact integral symplectic manifold, one can always find a complex line bundle $\mathbb{L}$ with connection $\nabla$ of curvature $\omega$ and the quantization $Q(M)$ is independent of this choice.

If $M$ is equipped with a Hamiltonian action of a torus $T$, the action can be lifted to $\mathbb{L}$ and the almost complex structure of $\mathbb{L}$ can be chosen to be $T$-invariant. In this case, the quantization $Q(M)$ is a finite-dimensional virtual $T$-module.

For $\xi \in \mathfrak{t}^*,$ denote by $M//_\xi T$ the reduced space of $M$ at $\xi$. For $\alpha$ a weight of $T$ and $V$ a virtual $T$-module, denote by $V^\alpha$ the sub-module of $V$ of weight $\alpha$.

The following result states that the component of weight $\alpha$ of the quantization of $M$ equals the quantization of the reduced space of $M$ at $\alpha$.

\begin{theorem}[Quantization commutes with reduction, \cite{Meinrenken96}]
Let $(M,\omega)$ is a compact integral symplectic manifold. Suppose $M$ is equipped with a Hamiltonian action of a torus $T$ and let $\alpha$ be a weight of $T$. Then
\begin{equation}
Q(M)^\alpha = Q(M//_\alpha T).
\label{eq:quantredsp}
\end{equation}
In other words,
\begin{equation}
Q(M) = \bigoplus_\alpha Q(M//_\alpha T) \alpha.
\label{eq:sumquantredsp}
\end{equation}
\label{th:QR}
\end{theorem}
Theorem \ref{th:QR} and equation (\ref{eq:sumquantredsp}) are valid only for regular values of the moment map of the Hamiltonian $T$-action. In the case where $\alpha$ is a singular value of the moment map, the singular quotient must be replaced by a slightly different construction using a shift of $\alpha$ \cite{Meinrenken96}. A similar caution applies in the case of Hamiltonian $T$-spaces which are non-compact and in the case of $b$-symplectic manifolds.

\subsubsection{Formal geometric quantization of non-compact Hamiltonian $T$-spaces}
In the case where $M$ is non-compact, equation (\ref{eq:quantredsp}) is used to \textit{define} the quantization of such Hamiltonian $T$-spaces.

\begin{definition}[Weitsman, \cite{Weitsman01}]
Let $M$ be a Hamiltonian $T$-space with integral symplectic form. Suppose the moment map for the $T$-action is proper. Let $V$ be an infinite-dimensional virtual $T$-module with finite multiplicities. We say
$$V= Q(M)$$
if for any compact Hamiltonian $T$-space $N$ with integral symplectic form, we have
\begin{equation}
(V\otimes Q(N))^T = Q((M \times N)//_0 T).
\label{eq:qreqnoncomp}
\end{equation}
In other words, as in (\ref{eq:sumquantredsp}),
$$Q(M) = \bigoplus_\alpha Q(M//_\alpha T) \alpha,$$
where the sum is taken over all weights $\alpha$ of $T.$
\label{def:fgq}
\end{definition}

The fact that the moment map is proper implies that the reduced space $(M \times N)//_0 T$ is compact for
any compact Hamiltonian $T$-space $N$, so that the right hand side of equation (\ref{eq:qreqnoncomp}) is well-defined.

\subsubsection{Formal geometric quantization of $b$-symplectic manifolds}

Suppose now that $(M,\omega)$ is compact, connected, oriented and not symplectic but $b$-symplectic. Suppose that it is equipped with a Hamiltonian action of a torus $T$, with nonzero leading modular weight. Let $\mathbb{L}$ be a complex line bundle on $M$ with connection $\nabla$ on $\mathbb{L}|_{M\setminus Z}$ whose curvature is $\omega|_{M\setminus Z}$.

In \cite{GuilleminMirandaWeitsman21}, the formal geometric quantization $Q(M)$ is defined as the following virtual $T$-module.

\begin{definition}
Let $V$ be a virtual $T$-module with finite multiplicities. We say
$$V=Q(M)$$
if for any compact Hamiltonian $T$-space $N$ with integral symplectic form, we have
\begin{equation}
(V\otimes Q(N))^T = \varepsilon Q((M \times N)//_0 T),\end{equation}
where $Q(N)$ denotes the standard geometric quantization of $N$, $Q(M \times N)//_0 T$ is the geometric quantization of the compact integral symplectic manifold $(M \times N)//_0 T$, and $\varepsilon$ is $+1$ if the symplectic orientation on the symplectic quotient $(M \times N)//_0 T$ agrees with the orientation inherited from $M \times N$ and $-1$ otherwise.
\label{eq:qreqb}
\end{definition}

This means that $Q(M)=Q(M \setminus Z)=\oplus_i \varepsilon_i Q((M \setminus Z)_i)$, where the $(M\setminus Z)_i$ are the connected components of $M \setminus Z$, $Q(M \setminus Z)$ is the formal geometric quantization of the non-compact Hamiltonian $T$-space $M \setminus Z$, and the $\varepsilon_i \in \{\pm 1\}$ are determined by the relative orientations of the symplectic forms on the components of $M \setminus Z$ and the overall orientation of $M$. Alternatively,
$$Q(M) = \bigoplus_\alpha \varepsilon(\alpha) Q(M//_\alpha T) \alpha,$$
where $Q(M//_\alpha T)$ must be defined using the shifting trick if $\alpha$ is not a regular value of the moment map, and each $\varepsilon(\alpha) \in \{ \pm 1\}$ is determined by the relative orientations of $M$ and $M//_{\alpha} T.$

In this $b$-symplectic case, the condition that the modular weight is non-zero guarantees that the reduced space $(M \times N)//_0 T$ is compact and symplectic for any compact Hamiltonian $T$-space $N$, so that $Q((M \times N)//_0 T)$ is well-defined.

\section{Bohr-Sommerfeld leaves via the moment map}
\label{section:BSleavesviaMM}

In this section we prove that, in an integral symplectic toric manifold, the Bohr-Sommerfeld set coincides with the set of integer points in the image of the moment map of the toric action. We prove that the same result also holds for integral $b$-symplectic toric manifolds.

First, let us recall a result from Guillemin and Sternberg \cite{GuilleminSternberg83} that identifies the Bohr-Sommerfeld leaves in a symplectic manifold using the moment map of an integrable system. In particular, it proves that the count of Bohr-Sommerfeld leaves of an integral symplectic manifold equals the count of the integer points in the image of the moment map.

\begin{theorem}[Guillemin-Sternberg, Theorem 2.4 in \cite{GuilleminSternberg83}]
Let $(M,\omega)$ be a $2n$-dimensional symplectic manifold endowed with an integrable system with moment map $\mu: M \to B$. Let $p$ and $q$ be two distinct points of $B$ contained in an open simply connected subset $B_0$ of $B$. Then:
\begin{itemize}
    \item There exists a globally defined system of action coordinates $f_1,\dots,f_n$ on $B_0$ such that $f_1(p) = \cdots =f_n(p) = 0$.
    \item If $p\in B$ is in the Bohr-Sommerfeld set, $q\in B$ is in the Bohr-Sommerfeld set if and only if $f_1(q),\dots,f_n(q)$ are integers.
\end{itemize} 
\label{thm:BSintegerpoints}
\end{theorem}

In Theorem \ref{thm:BSintegerpoints}, the correspondence between the Bohr-Sommerfeld leaves and the integer points of the moment map is established after the election of a globally defined system of action coordinates. Once a Bohr-Sommerfeld leaf is identified at a point $p\in B$, the $0$ of all the action coordinates is set there and the other Bohr-Sommerfeld leaves correspond to the integer points in these coordinates.

As a consequence, the integer condition that Bohr-Sommerfeld leaves have to satisfy can be shifted by an additive constant as long as it is the same constant for all the leaves, since the essential implication of Theorem \ref{thm:BSintegerpoints} is that the difference between the action variables at any two Bohr-Sommerfeld leaves is an integer. In view of this, a value of the moment map has to be fixed at some point (and leaf) of $M$ or, equivalently, a choice of the constant in the moment map has to be made.

We will prove that this choice of a constant in the moment map is also equivalent to the choice of the connection $1$-form $\Theta$ such that $d\Theta=\omega$.

\subsection{Dependence on the connection}
\label{sec:dependenceconnection}

In the following statements we show that we can always find a connection $1$-form $\Theta$ with $d\Theta=\omega$ such that the Bohr-Sommerfeld set coincides with the integer points in the image of the moment map in the appropriate coordinates.

\begin{lemma}
Let $(M,\omega,\mu)$ be a toric symplectic manifold. Let $\mathbb{L}$ be a complex line bundle over $M$ with connection $\nabla$ whose curvature is $\omega$. Given an arbitrary connection $1$-form $\Theta$, we can always produce a connection $\tilde \Theta$ that is invariant under the toric action.
\label{lem:invariantform}
\end{lemma}

\begin{proof}
Since the group $G=T^n$ acting on $M$ is compact, there exists a Haar measure $dg$ such that $\int_G dg=1$. Then, the averaging of a form $\Theta$ via $\int_G L_g^*\Theta dg$ provides a $G$-invariant form $\tilde \Theta$.
\end{proof}

\begin{proposition}
Let $(M,\omega,\mu)$ be a symplectic toric manifold. Let $\mathbb{L}$ be a complex line bundle over $M$ with connection $\nabla$ whose curvature is $\omega$. If $\Theta_1$ and $\Theta_2$ are two invariant connection $1$-forms, the function $\langle \Theta_1-\Theta_2,X \rangle$ is constant when $X$ is a vector field tangent to the polarization of $M$ by $\mu$.
\label{prop:differenceconnectionforms}
\end{proposition}

\begin{proof}
By definition, the connection $1$-forms $\Theta_i$ satisfy $\mu=\Theta_i(X)$ \cite{Kostant70}, where $\mu$ is a moment map of the toric action and $X$ is a vector field tangent to the polarization given by $\mu$.

Take $\alpha$ such that $\pi^*\alpha=\Theta_1-\Theta_2$. By Lemma \ref{lem:invariantform}, $\Theta_1$ and $\Theta_2$ can be chosen invariant so that $\mathcal{L}_X \alpha = 0$, since $\alpha$ is invariant under $X$. Then, by Cartan's magic formula, we have that $d\alpha(X) = di_X \alpha = \mathcal{L}_X \alpha - i_X d\alpha$.
 
If we have two invariant connection $1$-forms $\Theta_1$ and $\Theta_2$, we know that their difference $\Theta_1-\Theta_2=\pi^*\alpha$ is a constant. Then $\pi^*d\alpha=d\Theta_1-d\Theta_2=\omega-\omega=0$ and $d\alpha=0$ is zero.

Finally, if $d(\alpha(X)) = 0$, $\alpha(X)$ is a constant.
\end{proof}

\begin{remark}
In view of Proposition \ref{prop:differenceconnectionforms}, $\alpha(X_i)$ is a constant for any component $X_i$ of $X$. On the other hand, for any $\Theta_1,\Theta_2$, we have that $\Theta_1=\Theta_2+\pi^*\alpha$ and $\mu_1=\mu_2 +\pi^*\alpha(X)$. Then, fixing $\alpha(X_i)$ is equivalent to make a choice of the connection $1$-form and to fix the constant in the moment map.

In other words, the choice of a constant, which has to be made at some point in the process of quantization, can be made either by selecting a specific connection $1$-form or, equivalently, by setting at $0$ the coordinates of the moment map at a particular Bohr-Sommerfeld leaf.
\end{remark}

\subsection{The Bohr-Sommerfeld set in the image of the moment map}

Observe that any toric action on a symplectic manifold defines an integrable system. Then, in a symplectic toric manifold $(M,\omega,\mu)$, we can apply Theorem \ref{thm:BSintegerpoints} to identify the Bohr-Sommerfeld leaves of $M$ with the integer points in the image of the moment map.

We want to extend the correspondence between Bohr-Sommerfeld leaves and integer points in the image of the moment map to $b$-symplectic toric manifolds. For this reason, we prove first Theorem \ref{th:BSequalIntegerpoints}, which is a particular case of Theorem \ref{thm:BSintegerpoints} for symplectic toric manifolds. Then, we obtain Corollary \ref{cor:b-BSequalIntegerpoints}, the $b$-symplectic toric version of Theorem \ref{th:BSequalIntegerpoints}.

\begin{theorem}
Let $(M,\omega,\mu)$ be an integral symplectic toric manifold. Then, the Bohr-Sommerfeld set coincides with the integer points in the image of $\mu$.
\label{th:BSequalIntegerpoints}
\end{theorem}

\begin{proof}
Suppose $(M,\omega,\mu)$ is an integral symplectic toric manifold of a dimension $2n$. We will compute the Bohr-Sommerfeld leaves of $M$ with respect to the real Lagrangian polarization given by $\mu$ and see that each one of them is mapped to an integer point of $\Ima(\mu)$.

The Bohr-Sommerfeld leaves are the leafwise flat sections, i.e. the sections $s$ satisfying $\nabla_X s=0$, which are defined on an entire leaf.

We use the following formula from \cite{Kostant70} and \cite{DuistermaatGuilleminMeinrenkenWu95}: 
\begin{equation}
    D_X s = \nabla_X s + i\langle\mu,X\rangle s,
    \label{eq:flatconnection}
\end{equation}
with $X\in\mathfrak{t}$.

Since $M$ is a toric manifold, we can make use of the natural angle coordinates $\phi_1,\dots,\phi_n$ to obtain from (\ref{eq:flatconnection}) the following equation for each $1\leq i\leq n$:
\begin{equation}
    D_{\frac{\partial}{\partial \phi_i}} s = \nabla_{\frac{\partial}{\partial \phi_i}} s + i\langle\mu_i,\phi_i\rangle s.
    \label{eq:flatconnection2}
\end{equation}
At the leafwise flat sections, we have that $\nabla_{\frac{\partial}{\partial \phi_i}} s=0$ for each $i$, implying that 
\begin{equation}
    D_{\frac{\partial}{\partial \phi_i}} s =  i\langle\mu_i,\phi_i\rangle s.
    \label{eq:flatconnection3}
\end{equation}
We are looking for leaves of the foliation by $\mu$ which admit a flat section. Then, we can integrate (\ref{eq:flatconnection3}) with respect to $\phi$ from $0$ to $2\pi$ keeping $\mu_i$ constant, since it is constant along any leaf. We obtain 
\begin{equation}
    s(\phi_i=2\pi)=  e^{i\mu_i2\pi} \cdot s(\phi_i=0).
    \label{eq:flatconnection4}
\end{equation}
In order to be a well-defined section globally (on the entire leaf), $s(\phi_i=2\pi)$ has to equal $s(\phi_i=0)$. This condition is met if and only if $\mu_i\in\mathbb{Z}$. Then, the values of $\mu(M)$ at which there is a non-trivial leafwise flat section are in $\mathbb{Z}^n$.
\end{proof}

\begin{corollary}
Let $(^bM^{2n},Z,\omega,\mu)$ be an integral $b$-symplectic toric manifold. Then, the Bohr-Sommerfeld set coincides with the integer points in the image of $\mu$.
\label{cor:b-BSequalIntegerpoints}
\end{corollary}

\begin{proof}
By Proposition \ref{prop:localmomentmap}, in a neighbourhood of $L \times S^1 \times (-\varepsilon, \varepsilon) \cong U \subseteq M$ of $Z$, where $L$ is a leaf of its symplectic foliation the $T^n$-action on $U\setminus Z$ has moment map
\begin{equation*}
\mu_{U\setminus Z}: L \times S^1 \times \left((-\varepsilon, \varepsilon)\setminus\{0\}\right) \rightarrow \mathfrak{t}^* \cong \mathfrak{t}_{Z}^* \times \mathbb{R},\text{\,\,}
(\ell, \rho, t) \mapsto (\mu_{L}(\ell), c\log|t|),
\end{equation*}
where $c$ is the modular period of $Z$ and the map $\mu_{L}: L \rightarrow \mathfrak{t}_{Z}^*$ is a moment map for the $T_{Z}^{n-1}$-action on $L$.

As $M\setminus Z$ is an integral symplectic toric manifold, we can apply Theorem \ref{th:BSequalIntegerpoints}. Then, Bohr-Sommerfeld leaves of $M\setminus Z$ correspond to the points $(\mu_{L}(\ell), c\log|t|)$ such that $\mu_{L}(\ell)\in\mathbb{Z}^{n-1}$ and $c\log|t|\in \mathbb{Z}$.

The discreteness of the integrality condition on $c\log|t|$ implies that there are no Bohr-Sommerfeld leaves in the limit when $t$ tends to $0$. Thus, there are no Bohr-Sommerfeld leaves on $Z$ and this completes the proof of the Corollary.
\end{proof}

\section{Bohr-Sommerfeld quantization with sign for $b$-symplectic toric manifolds}
\label{sec:BSquantwithsign}
In this section we formalize the notion of counting Bohr-Sommerfeld leaves as a virtual vector spaces. We will redefine the Bohr-Sommerfeld quantization of Section \ref{section:preliminaries} in order to have it defines as a sum of $T$-modules. We will start doing it for the symplectic toric case, in which the quantization is already finite. Then, using the orientation of the manifold, we will define a quantization with sign which, together with the $T$-modules approach, will allow us to subtract virtual vector spaces.

To prove the equivalence between the Bohr-Sommerfeld quantization and the formal geometric quantization of an integral $b$-symplectic toric manifold, we need to introduce the \textit{Bohr-Sommerfeld quantization with sign for $b$-symplectic toric manifolds}. We define it for the $2$-dimensional $b$-symplectic sphere equipped with the rotation action. Then, we extend the definition to a general $b$-symplectic toric manifold.

Before the definition of Bohr-Sommerfeld quantization with sign, we will illustrate the need for it with a basic example.

\subsection{Bohr-Sommerfeld quantization via $T$-modules}

We start defining a \textit{ Bohr-Sommerfeld quantization} for symplectic toric manifolds. Assume that $(M,\omega,\mu)$ is a symplectic toric manifold and suppose $B_{BS}$ is the Bohr-Sommerfeld set of the polarization by the action of $T^n$.

To each $b\in B_{BS}$ regular value of $\mu$ (i.e., to each $b\in B_{BS}$ in the interior of the moment polytope $\Delta=\mu(M)$) we can associate a representation $\mathbb{C}(b)$ of $T^n$, where $b$ is the weight obtained by taking the quotient with the lifted action given by $\mu$. By the following Proposition, this representation $\mathbb{C}(b)$ is equal to the representation $\mathbb{C}\langle s_b\rangle)$ that appears in Equation \ref{eq:BSquantclassical}.

\begin{proposition}
For any $b=\mu(x)\in B_{BS}$ in the interior of $\Delta=\mu(M)$, $\mathbb{C}\langle s_b\rangle = \mathbb{C}(b)$.
\label{prop:leafequalsweight}
\end{proposition}

\begin{proof}
Suppose $b=\mu(x)\in B_{BS}$. Then, $s_b$ is a section of the line bundle over $M$ such that it satisfies Equation \ref{eq:flatconnection}, namely, $$D_X s_b = \nabla_X s_b + i\langle\mu,X\rangle s_b.$$

If $b$ is in the interior of $\Delta=\mu(M)$, it has a neighbourhood $U$ in which all points are regular values of $\mu$ or, equivalently, in the pre-image $\mu^{-1}(U)$ the toric action has no singularities. Then, on the pre-images $x=\mu^{-1}(b)$ of these values, the line bundle $\mathbb{L}$ restricts to a line bundle with connection $\mathbb{L} \to \mu^{-1}(b)$, where $\mu^{-1}(b)$ is a torus $T^n$. Then, the connection on this line bundle, since $\mu(x)$ is integral, is given by the line $\mathbb{C}(\mu(x))$, where $\mathbb{C}(\mu(x))$ is the quotient $\mathbb{L}_{\mu^-{1}(\mu(x))}/T^n$ \cite{Kostant70}.
\end{proof}

In the case of a point $b$ in the boundary of the moment polytope $\Delta$, the naive quotient does not give the correct result, so it has to be defined slightly differently, although it will formally appear to be the same. What we do is to  consider that all points in the boundary of $\Delta$ give weights of a one-dimensional representation $\mathbb{C}(w)$ of $\mathbb{R}^n$.

To each $b\in B_{BS}$ singular value of $\mu$ (i.e., to each $b\in B_{BS}$ in the boundary of the moment polytope $\mu(M)$), associate a representation  of $T^n$ in a more sophisticated way. A version of the shifting trick (\cite{Meinrenken96}, see also \cite{Weitsman01}) shows that, although we do not have a proper action of $T^n$ in $\mu^{-1}(b)$, if we take a succession ${a_i}$ in the interior $\mu(M)$ with limit at $b$, the quotient construction gives representations of $\mathbb{R}^n(a_i)$ approaching $\mathbb{C}(b)$. This is because the shifting trick can be applied to any regular value of $\mu$, in particular to all the $a_i\in \int \mu(M)$. Hence, at all these values the quotient can be identified with the quotient space at $0$. Finally, by continuity and since $\mathbb{C}(b)$ itself is a $T^n$-representation because $b$ is an integral value, it is well-defined.

We have associated each $b\in B_{BS}$ a representation $\mathbb{C}(b)$ that depends on whether $b$ is in the interior or in the boundary of $\Delta$ and we can define the \textit{Bohr-Sommerfeld quantization via $T$-modules} as follows.

\begin{definition}
The Bohr-Sommerfeld quantization via $T$-modules of an compact integral symplectic toric manifold is
\begin{equation}
    Q(M)=\bigoplus_{b\in B_{BS}}\mathbb{C}(b).
    \label{eq:equivBSquant}
\end{equation}
\label{def:equivariantBS}
\end{definition}

\begin{remark}
Observe that $Q(M)$ can also be defined directly as $=\bigoplus_{b \in \Delta \cap Z^n}\mathbb{C}(b)$. The sum $\bigoplus_{b \in \Delta \cap Z^n}\mathbb{C}(b)$ in Equation \ref{eq:equivBSquant} can be an infinite-dimensional module if $M$ is a non-compact toric manifold (in particular, if $M$ is a $b$-symplectic toric manifold) because $\Delta$ may be unbounded. Nevertheless, in all cases each weight has finite multiplicity (it is either zero or one).
\label{rem:quantcanbeinfinite}
\end{remark}

By the previous remark, if we apply Definition \ref{def:equivariantBS} to a $b$-symplectic toric manifold we will obtain an infinite-dimensional quantization space. For this reason, we introduce the \textit{Bohr-Sommerfeld quantization with sign} below. We will do it via $T$-modules again.

\subsection{The motivating example, the Bohr-Sommerfeld of the canonical $b$-sphere}

The simplest example of a $b$-symplectic toric manifold is the $b$-sphere with the singular hypersurface $Z$ being a single circle at the equator and endowed with the action of rotation around its vertical axis. We are going to see that its Bohr-Sommerfeld quantization via $T$-modules gives an infinite-dimensional space and that, on the other hand, its Bohr-Sommerfeld quantization with sign gives a finite-dimensional space. Besides, in Section \ref{section:equivalencegqfgq} we will prove that the latter quantization coincides with the formal geometric quantization in the general case.

Consider the $b$-symplectic sphere $(^bS^2,Z,\omega,\mu)$, with the hypersurface $Z$ on the the equator and $\mu$ the moment map of the rotation around the vertical axis. Away from the poles, take cylindrical polar coordinates $\{(h,\theta)\colon -1< h < 1, 0\leq \theta<2\pi\}$. $Z$ corresponds to $\{h=0\}$, the $b$-symplectic form on $^bS^2$ is $\omega=\frac{1}{h}dh\wedge d\theta=d\log{\vert h\vert }\wedge d\theta$ and the moment map writes as the $b$-function $\mu=-\log|h|$.

Let $\mathbb{L}$ be a complex line bundle on $^bS^2$ with connection $\nabla$ defined by $$\nabla_X \sigma =X(\sigma) - \sigma i \log{\vert h\vert }\,d\theta(X).$$ It has curvature $\omega$ and makes $\mathbb{L}$ into a pre-quantization line bundle.

Take the real polarization $P$ of $^bS^2$ given by $\mu=-\log|h|$, i.e., the polarization spanned by $\frac{\partial}{\partial\theta}$ at each point of $^bS^2$. The leaves of $P$ are either circles of the form $\{h_0\}\times S^1$, with $h_0\in (-1,0)\cup (0,1)$, or the two poles.

The leaf-wise flat sections, which satisfy $\nabla_X \sigma = 0$ for $X$ tangent to the polarization $P$, are of the form $$\sigma(h,\theta)=a(h)e^{i\log{\vert h\vert }\theta},$$ with $a(h)\in\mathbb{C}$ (see for instance \cite{Hamilton10} or \cite{MirMiranda21} for the explicit computations). The Bohr-Sommerfeld leaves on $^bS^2\setminus Z$ are the leaves of the foliation by $P$ that admit a non-trivial global leaf-wise flat section $\sigma$.

Along each leaf $\{h_0\}\times S^1$ of the foliation by $P$ of $^bS^2\setminus Z$, $h$ is fixed at $h_0$. Then, a leaf is Bohr-Sommerfeld if it admits a section $\sigma(h_0,\theta)$ such that $$\sigma(h_0,\theta)=\sigma(h_0,\theta+2\pi).$$

Therefore, $\{h_0\}\times S^1$ is a Bohr-Sommerfeld leaf if $1=e^{2\pi i\log{\vert h_0\vert}}$ or, equivalently, if $\log{\vert h_0\vert }\in\mathbb{Z}$. And the set of all Bohr-Sommerfeld leaves of the foliation by $P$ of $^bS^2\setminus Z$ is $$B_{BS}=\{\{e^{-m}\}\times S^1\subset \,^bS^2\setminus Z\colon m \in \mathbb{N}\}\,\bigcup\,\{\{-e^{-m}\}\times S^1\subset \,^bS^2\setminus Z\colon m \in \mathbb{N}\}.$$

The  Bohr-Sommerfeld quantization of $(^bS^2,Z,\omega,\mu)$ is, by Definition \ref{def:equivariantBS}, the following sum
$$Q(^bS^2) = \bigoplus_{b\in B_{BS}} \mathbb{C}(b) = \bigoplus_{2\mathbb{N}} \mathbb{C}(b),$$ which is an infinite-dimensional space.

Observe that the quantization is infinite-dimensional because there is an infinite number of Bohr-Sommerfeld leaves arbitrarily close to $Z$ both in the upper and the lower hemisphere. Explicitly, for any $a>0$, there is an infinite number of values of $h\in(0,a)$ and also of $h\in(-a,0)$ satisfying the condition $\log\vert h\vert \in \mathbb{Z}$ of a Bohr-Sommerfeld leaf.

\subsection{Bohr-Sommerfeld quantization with sign for the $b$-symplectic toric sphere}

In order to obtain a finite-dimensional quantization space for $(^bS^2,Z=\{h_z\}\times S^1,\omega,\mu)$, we will define a way of summing the quantization spaces corresponding to the Bohr-Sommerfeld leaves of $^bS^2$ that takes into account the orientation of the hemisphere to which each Bohr-Sommerfeld leaf belongs. The intention is to define the quantization space by "adding" the virtual vector space of Bohr-Sommerfeld leaves at the northern hemisphere and to "subtracting" the virtual vector space of Bohr-Sommerfeld at the southern hemisphere. In such a way, the final sum will be a finite-dimensional vector space.

This idea can be adapted formally by defining the \textit{Bohr-Sommerfeld quantization with sign} of $(^bS^2,Z=\{h_z\}\times S^1,\omega,\mu)$, with $-1<h_z<1$, as the sum of the virtual vector spaces $\mathbb{C}(b)$ associated to the Bohr-Sommerfeld leaves in $S^2_+=(h_z,1)\times S^1\subset S^2\setminus Z$ \textit{minus} the sum of as many copies of the virtual vector spaces $\mathbb{C}(b)$ associated to the Bohr-Sommerfeld leaves in $S^2_-=(-1,h_z)\times S^1\subset S^2\setminus Z$.

\begin{definition}
Let $B_{BS}$ be the Bohr-Sommerfeld set of $(^bS^2,Z=\{h_z\}\times S^1,\omega,\mu)$. For each $b\in B_{BS}$, define $\epsilon(b)$ as $\epsilon(b)=+1$ if $\mu^{-1}(b)$ is a Bohr-Sommerfeld leaf in the northern hemisphere $S^2_+$ and $\epsilon(b)=-1$ if $\mu^{-1}(b)$ is a Bohr-Sommerfeld leaf in the southern hemisphere $S^2_-$. We call $\epsilon(b)$ the \textit{sign of $b$}.
\label{def:signBS1}
\end{definition}

\begin{definition}[Bohr-Sommerfeld quantization with sign of $(^bS^2,Z=\{h_z\}\times S^1,\omega,\mu)$]
Let $B_{BS}$ be the Bohr-Sommerfeld set of $(^bS^2,Z=\{h_z\}\times S^1,\omega,\mu)$. The \textit{quantization with sign of $(^bS^2,Z=\{h_z\}\times S^1,\omega,\mu)$} is $$\tilde{Q}(^bS^2)=\bigoplus_{b\in B_{BS}} \epsilon(b) \mathbb{C}(b).$$
\label{def:geomquantwithsign1}
\end{definition}

\begin{lemma}
$\tilde{Q}(^bS^2)$ is a finite-dimensional vector space.
\label{lem:quantwithsignisfinite}
\end{lemma}

\begin{proof}
First, observe that $\bigoplus_{b\in B_{BS}} \epsilon(b) \mathbb{C}(b)$ is an infinite-dimensional module with finite multiplicities (which may be negative).

For each Bohr-Sommerfeld leaf of the form $\{h_z+\delta\}\times S^1$ in $S^2_+$, there is a Bohr-Sommerfeld leaf of the form $\{h_z-\delta\}\times S^1$ in $S^2_-$ for any $\delta>0$ small enough. Then, at any symmetric neighbourhood $U$ of $Z$ in $^bS^2\setminus Z$, the virtual module $\bigoplus_{b\in B_{BS}\cap \mu(U)} \epsilon(b) \mathbb{C}(b)$ is exactly $0$. 

On the other hand, there are only finitely many Bohr-Sommerfeld leaves in $^bS^2\setminus U$ and, therefore, $\bigoplus_{b\in B_{BS}\cap \mu(^bS^2\setminus U)}(b) \epsilon(b) \mathbb{C}$ is finite-dimensional. Hence, $\tilde{Q}(^bS^2)$ is a finite-dimensional vector space.
\end{proof}

The Bohr-Sommerfeld quantization with sign of the $b$-symplectic toric sphere with $Z$ being the circle at the equator is exactly $0$ due to its symmetry (see Figure \ref{fig:BSbsphere}). In general, $\tilde{Q}(^bS^2)$ is not $0$ because $Z$ does not need to be the at equator and there is not a one-to-one correspondence between Bohr-Sommerfeld leaves at each side of $Z$.
 
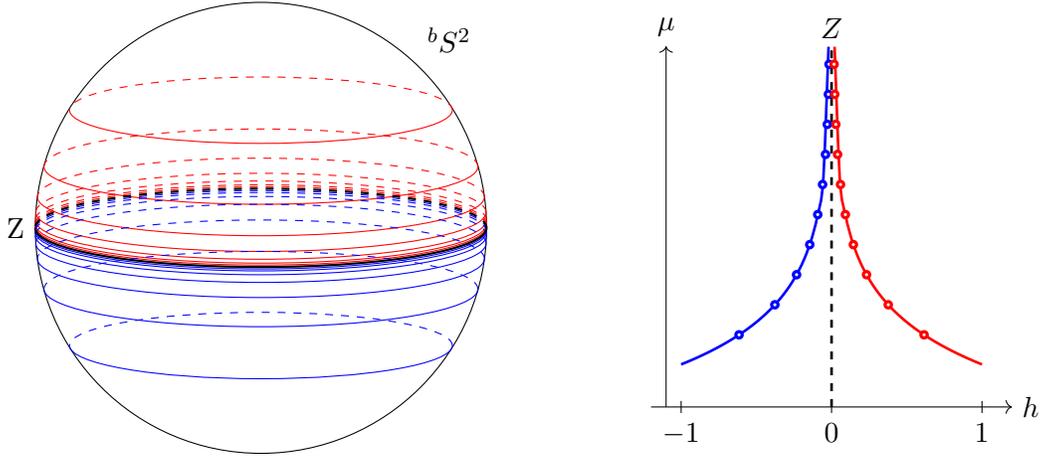
\begin{figure}[ht!]
\centering
\begin{tikzpicture}[scale=3]
    \def\phi{10};
    \draw (0, 0) circle (1) node[xshift=2.5cm,yshift=2.5cm] {$^bS^2$};
    \foreach \latitude in {0.5, 1, 2, 4, 8, 16, 32} {
    \pgfmathsetmacro\verticaloffset{cos(\phi)*sin(\latitude)};
    \pgfmathsetmacro\radius{cos(\latitude)};
    \tikzset{xyplane/.estyle={cm={1, 0, 0, cos(90 + \phi), (0, \verticaloffset)}}}
    \draw [xyplane, red] (\radius,0) arc (0:180:\radius);
    \draw [xyplane, dashed, red] (\radius,0) arc (360:180:\radius);
    }
    \foreach \latitude in {-0.5, -1, -2, -4, -8, -16, -32} {
    \pgfmathsetmacro\verticaloffset{cos(\phi)*sin(\latitude)};
    \pgfmathsetmacro\radius{cos(\latitude)};
    \tikzset{xyplane/.estyle={cm={1, 0, 0, cos(90 + \phi), (0, \verticaloffset)}}}
    \draw [xyplane, blue] (\radius,0) arc (0:180:\radius);
    \draw [xyplane, dashed, blue] (\radius,0) arc (360:180:\radius);
    }
    \tikzset{plane/.estyle={cm={1, 0, 0, cos(90 + \phi), (0, 0)}}}
    \draw [thick,plane] (1,0) arc (0:180:1);
    \draw [thick,plane,dashed] (1,0) arc (360:180:1) node[left] {Z};
\end{tikzpicture}\hspace{2cm}
\begin{tikzpicture}[scale=0.4,xscale=0.5,yscale=2]
  \draw[->] (-12, -3) -- (12, -3) node[right] {$h$};
  \draw (-10, -2.9) -- (-10, -3.1) node[below] {$-1$};
  \draw (0, -2.9) -- (0, -3.1) node[below] {$0$};
  \draw (10, -2.9) -- (10, -3.1) node[below] {$1$};
  \draw[line width = 1pt, domain=-10:-0.13, smooth, variable=\h, blue] plot ({\h}, {(-ln(abs(\h+0.1)))});
  \draw[line width = 1pt, black, dashed] (0, -3) -- (0, 3);
  \draw[line width = 1pt, domain=0.13:10, smooth, variable=\h, red] plot ({\h}, {(-ln(abs(\h-0.1))});
  \draw [white,fill=white] (-1,3) rectangle (1,3.5);
  \node at (0, 3.3) (Z1) {$Z$};
  \draw[->] (-11, -3) -- (-11, 3) node[above] {$\mu$};
\foreach \bs in {-1.8,-1.3,-0.8,-0.3,0.2,0.7,1.2,1.7,2.2,2.7} {
    \draw[very thick, blue,xscale=2,yscale=0.5,fill=white] ({(-exp(-(\bs))-0.1)*0.5},{\bs*2}) circle (3pt);
    }
\foreach \bs in {-1.8,-1.3,-0.8,-0.3,0.2,0.7,1.2,1.7,2.2,2.7} {
    \draw[very thick, red,xscale=2,yscale=0.5,fill=white] ({(exp(-(\bs))+0.1)*0.5},{\bs*2}) circle (3pt);
    }
\end{tikzpicture}
\caption{On the left, Bohr-Sommerfeld leaves on the northern hemisphere (in red) and the southern hemisphere (in blue) of $(^bS^2,Z=\{h_0=0\}\times S^1,\omega,\mu)$. On the right, the moment map $\mu=-\log|h|$ with dots indicating Bohr-Sommerfeld leaves.}
\label{fig:BSbsphere}
\end{figure}

\subsection{Bohr-Sommerfeld quantization with sign for $b$-symplectic toric surfaces}

We can naturally generalize the definition of Bohr-Sommerfeld quantization with sign of the $b$-symplectic toric sphere $(^bS^2,Z=\{h_z\}\times S^1,\omega,\mu)$ to any $b$-symplectic toric surface $(M^2,Z,\omega,\mu)$.

By Theorem \ref{thm:b-surfaces}, a $b$-symplectic toric surface $(M^2,Z,\omega,\mu)$ is equivariantly $b$-symplectomorphic to either $(S^2,Z)$ or $(T^2,Z)$, where $Z$ is a collection of latitude circles (in $T^2$, an even number of them) and $\mu$ is the standard rotation.

Since $Z$ defines an orientation in each connected component of $M^2\setminus Z$, we can associate a sign to each Bohr-Sommerfeld leaf depending on the orientation of the component to which it belongs.

\begin{definition}
Let $B_{BS}$ be the Bohr-Sommerfeld set of $(M^2,Z,\omega,\mu)$. For each $b\in B_{BS}$, define $\epsilon(b)$ as $\epsilon(b)=+1$ if $\mu^{-1}(b)$ belongs to a positively-oriented connected component and as $\epsilon(b)=-1$ if $\mu^{-1}(b)$ belongs to a negatively-oriented one. We call $\epsilon(b)$ the \textit{sign of $b$}.
\label{def:signBS2}
\end{definition}

\begin{definition}[Bohr-Sommerfeld quantization with sign of $(M^2,Z,\omega,\mu)$]
Let $B_{BS}$ be the Bohr-Sommerfeld set of $(M^2,Z,\omega,\mu)$. The \textit{quantization with sign of $(M^2,Z,\omega,\mu)$} is $$\tilde{Q}(M^2)=\bigoplus_{b\in B_{BS}} \epsilon(b) \mathbb{C}(b).$$
\label{def:geomquantwithsign2}
\end{definition}

\begin{lemma}
$\tilde{Q}(M^2)$ is a finite-dimensional vector space.
\label{lem:quantwithsignisfinite2}
\end{lemma}

\begin{proof}
Take a symmetric neighbourhood $U\subset M^2\setminus Z$ of $Z$, which always exists by Proposition \ref{prop:localmomentmap}. By the same argument in the proof of Lemma \ref{lem:quantwithsignisfinite}, the sum $\bigoplus_{b\in B_{BS}} \epsilon(b) \mathbb{C}(b)$ is $0$ at $U\setminus Z$ and it is a finite-dimensional vector space in $M^2\setminus U$. Hence, $\tilde{Q}(M^2)$ is finite-dimensional.
\end{proof}

This definition allows to obtain, from a symplectic toric surface that has an infinite number of Bohr-Sommerfeld, a finite-dimensional quantization space (see Figures \ref{fig:bsphere6} and \ref{fig:momentmapbspheregeneraldetalls}).

\begin{figure}[ht!]
\begin{center}
\begin{tikzpicture}[scale=4]
    \def\phi{10};
    \draw (0, 0) circle (1);
    \foreach \latitude [count=\i] in {-53,-8,8,26,69} {
    \pgfmathsetmacro\verticaloffset{1*cos(\phi)*sin(\latitude)};
    \pgfmathsetmacro\radius{cos(\latitude)};
    \tikzset{xyplane/.estyle={cm={1, 0, 0, cos(90 + \phi), (0, \verticaloffset)}}}
    \draw [xyplane, black, thick] (\radius,0) arc (0:180:\radius);
    \draw [xyplane, dashed, black, thick] (\radius,0) arc (360:180:\radius) node[left] {$Z_{\i}$};
    }
    \foreach \latitude in {-65, -57, -55,-54, -53.5, -7.5, -7, -6, -4, 0, 4, 6, 7, 7.5, 26.5, 27, 28, 30, 34, 44, 55, 63, 67, 68, 68.5} {
    \pgfmathsetmacro\verticaloffset{cos(\phi)*sin(\latitude)};
    \pgfmathsetmacro\radius{cos(\latitude)};
    \tikzset{xyplane/.estyle={cm={1, 0, 0, cos(90 + \phi), (0, \verticaloffset)}}}
    \draw [xyplane, blue] (\radius,0) arc (0:180:\radius);
    \draw [xyplane, dashed, blue] (\radius,0) arc (360:180:\radius);
    }
    \foreach \latitude in {-52.5, -52, -51, -49, -45, -37, -24, -16, -12, -10, -9, -8.5, 8.5, 9, 10, 12, 17, 22, 24, 25, 25.5, 69.5, 70, 71, 77} {
    \pgfmathsetmacro\verticaloffset{cos(\phi)*sin(\latitude)};
    \pgfmathsetmacro\radius{cos(\latitude)};
    \tikzset{xyplane/.estyle={cm={1, 0, 0, cos(90 + \phi), (0, \verticaloffset)}}}
    \draw [xyplane, red] (\radius,0) arc (0:180:\radius);
    \draw [xyplane, dashed, red] (\radius,0) arc (360:180:\radius);
    }
\end{tikzpicture}
\end{center}
\caption{A $b$-symplectic toric sphere with $Z$ made of 5 latitude circles.}
\label{fig:bsphere6}
\end{figure}
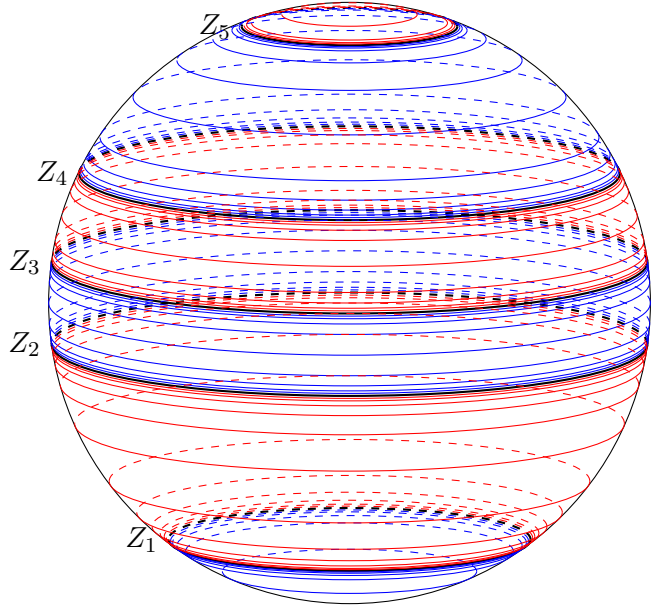

\begin{figure}[ht!]
\begin{center}
\begin{tikzpicture}[scale=0.7]
  \def\ha{0};\def\hb{3}\def\hc{8}\def\hd{10}\def\he{12}\def\hf{16}\def\hg{17};\def\dh{0.1};\def\ddh{0.01};
  \def\mua{-ln(abs((-\ha+\hb)/2-\dh))-ln(abs((\ha-\hb)/2+\dh))+0.0001};
  \draw[->] (\ha, -2) -- (\hg+1, -2) node[right] {$h$};
  \draw[line width = 1pt, domain=(\ha+\hb)/2:\hb-\dh-1.5*\ddh), smooth, variable=\h, blue] plot ({\h}, {(-ln(abs(\h-\ha-\dh))-ln(abs(\h-\hb+\dh)))});
  \foreach \bs in {0,1,2,3} {
    \draw[very thick,blue,fill=white] ({((\ha+\hb)/2+sqrt((\ha+\hb)^2/4-(\ha+\dh)*(\hb-\dh)-exp(-(\mua+\bs)))},{(\mua+\bs)}) circle (3pt);
    }
  \draw[line width = 1pt, black, dashed] (\hb, -2) -- (\hb,3);
  \draw[line width = 1pt, domain=\hb+\dh+\ddh:\hc-\dh-\ddh, smooth, variable=\h, red] plot ({\h}, {(-ln(abs(\h-\hb-\dh))-ln(abs(\h-\hc+\dh)))+0.3});
  \foreach \bs in {0,1,2,3} {
    \draw[very thick,red,fill=white] ({(\hb+\hc)/2-sqrt((\hb+\hc)^2/4-(\hb+\dh)*(\hc-\dh)-exp(-(\mua+\bs-0.3)))},{(\mua+\bs)}) circle (3pt);
    }
  \foreach \bs in {1,2,3} {
    \draw[very thick,red,fill=white] ({(\hb+\hc)/2+sqrt((\hb+\hc)^2/4-(\hb+\dh)*(\hc-\dh)-exp(-(\mua+\bs-0.3)))},{(\mua+\bs)}) circle (3pt);
    }
  \foreach \bs in {0} {
    \draw[very thick,red,fill=red] ({(\hb+\hc)/2+sqrt((\hb+\hc)^2/4-(\hb+\dh)*(\hc-\dh)-exp(-(\mua+\bs-0.3)))},{(\mua+\bs)}) circle (3pt);
    }
  \draw[line width = 1pt, black, dashed] (\hc, -2) -- (\hc,3);
  \draw[line width = 1pt, domain=\hc+\dh+2*\ddh:\hd-\dh-2*\ddh, smooth, variable=\h, blue] plot ({\h}, {(-ln(abs(\h-\hc-\dh))-ln(abs(\h-\hd+\dh)))});
  \foreach \bs in {1,2,3} {
    \draw[very thick,blue,fill=white] ({(\hc+\hd)/2-sqrt((\hc+\hd)^2/4-(\hc+\dh)*(\hd-\dh)-exp(-(\mua+\bs)))},{(\mua+\bs)}) circle (3pt);
    \draw[very thick,blue,fill=white] ({(\hc+\hd)/2+sqrt((\hc+\hd)^2/4-(\hc+\dh)*(\hd-\dh)-exp(-(\mua+\bs)))},{(\mua+\bs)}) circle (3pt);
    }
  \draw[line width = 1pt, black, dashed] (\hd, -2) -- (\hd,3);
  \draw[line width = 1pt, domain=\hd+\dh+\ddh:\he-\dh-\ddh, smooth, variable=\h, red] plot ({\h}, {(-ln(abs(\h-\hd-\dh))-ln(abs(\h-\he+\dh)))-1});
  \foreach \bs in {1,2,3} {
    \draw[very thick,red,fill=white] ({(\hd+\he)/2-sqrt((\hd+\he)^2/4-(\hd+\dh)*(\he-\dh)-exp(-(\mua+\bs+1)))},{(\mua+\bs)}) circle (3pt);
    }
  \foreach \bs in {0} {
    \draw[very thick,red,fill=red] ({(\hd+\he)/2-sqrt((\hd+\he)^2/4-(\hd+\dh)*(\he-\dh)-exp(-(\mua+\bs+1)))},{(\mua+\bs)}) circle (3pt);
    }
  \foreach \bs in {1,2,3} {
    \draw[very thick,red,fill=white] ({(\hd+\he)/2+sqrt((\hd+\he)^2/4-(\hd+\dh)*(\he-\dh)-exp(-(\mua+\bs+1)))},{(\mua+\bs)}) circle (3pt);
    }
  \foreach \bs in {0} {
    \draw[very thick,red,fill=red] ({(\hd+\he)/2+sqrt((\hd+\he)^2/4-(\hd+\dh)*(\he-\dh)-exp(-(\mua+\bs+1)))},{(\mua+\bs)}) circle (3pt);
    }
  \draw[line width = 1pt, black, dashed] (\he, -2) -- (\he,3);
  \draw[line width = 1pt, domain=\he+\dh+3*\ddh:\hf-\dh-3*\ddh, smooth, variable=\h, blue] plot ({\h}, {(-ln(abs(\h-\he-\dh))-ln(abs(\h-\hf+\dh)))+1});
  \foreach \bs in {1,2,3} {
    \draw[very thick,blue,fill=white] ({(\he+\hf)/2-sqrt((\he+\hf)^2/4-(\he+\dh)*(\hf-\dh)-exp(-(\mua+\bs-1)))},{(\mua+\bs)}) circle (3pt);
    }
  \foreach \bs in {2,3} {
    \draw[very thick,blue,fill=white] ({(\he+\hf)/2+sqrt((\he+\hf)^2/4-(\he+\dh)*(\hf-\dh)-exp(-(\mua+\bs-1)))},{(\mua+\bs)}) circle (3pt);
    }
  \foreach \bs in {1} {
    \draw[very thick,blue,fill=blue] ({(\he+\hf)/2+sqrt((\he+\hf)^2/4-(\he+\dh)*(\hf-\dh)-exp(-(\mua+\bs-1)))},{(\mua+\bs)}) circle (3pt);
    }
  \draw[line width = 1pt, black, dashed] (\hf, -2) -- (\hf,3);
  \draw[line width = 1pt, domain=\hf+\dh+\ddh:\hf.5, smooth, variable=\h, red] plot ({\h}, {(-ln(abs(\h-\hf-\dh))-ln(abs(\h-\hg+\dh)))-1.5});
  \foreach \bs in {2,3} {
    \draw[very thick,red,fill=white] ({(\hf+\hg)/2-sqrt((\hf+\hg)^2/4-(\hf+\dh)*(\hg-\dh)-exp(-(\mua+\bs+1)))},{(\mua+\bs)}) circle (3pt);
    }
  \draw [white,fill=white] (0.01,3) rectangle (\hg,3.5);
  \node at (3, 3.3)   (Z1) {$Z_1$};
  \node at (8, 3.3)   (Z2) {$Z_2$};
  \node at (\hd, 3.3)   (Z3) {$Z_3$};
  \node at (\he, 3.3)   (Z4) {$Z_4$};
  \node at (\hf, 3.3)   (Z5) {$Z_5$};
\draw (1.5, -2.0) -- (1.5, -2.1) node[below] {$-1$};
\draw (3, -2.0) -- (3, -2.1) node[below] {$h_{Z_1}$};
\draw (8, -2.0) -- (8, -2.1) node[below] {$h_{Z_2}$};
\draw (\hd, -2.0) -- (\hd, -2.1) node[below] {$h_{Z_3}$};
\draw (\he, -2.0) -- (\he, -2.1) node[below] {$h_{Z_4}$};
\draw (\hf, -2.0) -- (\hf, -2.1) node[below] {$h_{Z_5}$};
\draw (16.5, -2.0) -- (16.5, -2.1) node[below] {$1$};
\draw[->] (0, -2) -- (0, 3) node[above] {$\mu$};
\end{tikzpicture}\end{center}
\caption{The moment map of the $b$-sphere of Figure \ref{fig:bsphere6}. Blue dots are the Bohr-Sommerfeld leaves in positively oriented components of $S^2\setminus Z$. Red dots are the Bohr-Sommerfeld leaves in negatively oriented components of $S^2\setminus Z$. White-filled dots represent Bohr-Sommerfeld leaves at the neighbourhood of each $Z_i$ whose associated quantization spaces symmetrically cancel.}
\label{fig:momentmapbspheregeneraldetalls}
\end{figure}
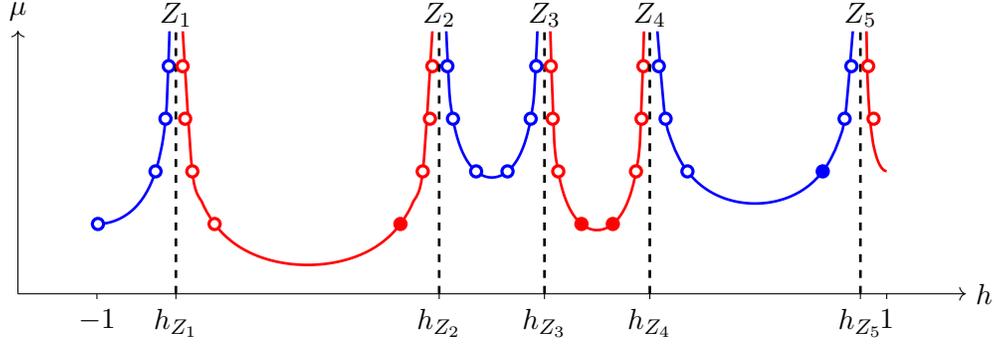

\subsection{Bohr-Sommerfeld quantization with sign for $b$-symplectic toric manifolds}

Proposition \ref{prop:bdelzant} states that a general $b$-symplectic toric manifold $(M^{2n},Z,\omega,\mu)$ decomposes either into the product of a $b$-symplectic toric $(^bT^2,Z_T,\omega_T,\mu_T)$ with a classic symplectic toric manifold, or it can be obtained from the product of a $b$-symplectic toric $(^bS^2,Z_S,\omega_S,\mu_S$ with a classic symplectic toric manifold by a sequence of symplectic cuts away from $Z$.

$Z$ defines an orientation in the connected components of $^bT^2\setminus Z_T$ and $^bS^2\setminus Z_S$. Then, we can give the definition of Bohr-Sommerfeld quantization with sign for $b$-symplectic manifolds based on the orientation of the connected components of $^bS^2$ or $^bT^2$ to which each Bohr-Sommerfeld leaf of $M$ projects.

\begin{definition}
Let $B_{BS}$ be the Bohr-Sommerfeld set of $(M^{2n},Z,\omega,\mu)$. Let $\pi$ be the canonical projection of $M$ to $^bS^2$ or $^bT^2$. For each $b\in B_{BS}$, define $\epsilon(b)$ as $\epsilon(b)=+1$ if $\pi\left(\mu^{-1}(b)\right)$ belongs to a positively-oriented connected component and as $\epsilon(b)=-1$ if $\pi\left(\mu^{-1}(b)\right)$ belongs to a negatively-oriented one. We call $\epsilon(b)$ the \textit{sign of $b$}.
\label{def:signBS3}
\end{definition}

\begin{definition}[Bohr-Sommerfeld quantization with sign of $(M^{2n},Z,\omega,\mu)$]
Let $B_{BS}$ be the Bohr-Sommerfeld set of $(M^{2n},Z,\omega,\mu)$. The \textit{quantization with sign of $(M^{2n},Z,\omega,\mu)$} is $$\tilde{Q}(M^{2n})=\bigoplus_{b\in B_{BS}} \epsilon(b) \mathbb{C}(b).$$
\label{def:geomquantwithsign3}
\end{definition}

\begin{lemma}
$\tilde{Q}(M^{2n})$ is a finite-dimensional vector space.
\label{lem:quantwithsignisfinite3}
\end{lemma}

\begin{proof}
Take a symmetric neighbourhood $U\subset M^{2n}\setminus Z$, which always exists by Proposition \ref{prop:localmomentmap}. Project $U$ by $\pi$ to $^bS^2$ or $^bT^2$ and use the same argument in the proof of Lemmas \ref{lem:quantwithsignisfinite} and \ref{lem:quantwithsignisfinite2} to obtain that $\tilde{Q}(M^{2n})$ is finite-dimensional.
\end{proof}

\begin{remark}
Observe that in the definition of the Bohr-Sommerfeld quantization with sign and, specifically, in the sum $\bigoplus_{b\in B_{BS}} \epsilon(b)\mathbb{C}(b)$, we are using the fact that we have a group $T$ acting with weights with finite multiplicity, thus each $\mathbb{C}(b)$ is finite dimensional. Then, the infinite sum $\bigoplus_{b\in B_{BS}} \epsilon(b)\mathbb{C}(b)$ is of finite dimension.

To illustrate the behaviour of this sum, consider the abelian case where all representations $\mathbb{C}(w)$ are one-dimensional and correspond to a weight $w$, so that $A = \oplus_w m_w \mathbb{C}(w)$ and $B = \oplus_w n_w \mathbb{C}(w)$, where $m_w$ and $n_w$ are integers. Then, one can consider $A-B$ to be $\oplus_w (m_w  - n_w) \mathbb{C}(w)$, which is a virtual $T$-module with finite multiplicities. The dimension of $A-B$ is $d=\sum_w (m_w - n_w)$ if this sum is finite and, then, one can talk about a finite-dimensional virtual module of dimension $d$.
\end{remark}

\section{The final count. Bohr-Sommerfeld quantization equals formal geometric quantization}
\label{section:equivalencegqfgq}

In this section we compare Bohr-Sommerfeld quantization with formal geometric quantization. We prove that they are equivalent because they both equal the count of integer points in the image of the moment map. We do it both for the symplectic case, in Theorem \ref{thm:GQcoincidesFGQ}, and for the $b$-symplectic case, in Theorem \ref{thm:b-GQcoincidesFGQ}.

\begin{theorem}
Let $(M^2,\omega,\mu)$ be a symplectic toric manifold. Then, the formal geometric quantization of $M$ coincides with the Bohr-Sommerfeld quantization.
\label{thm:GQcoincidesFGQ}
\end{theorem}

\begin{proof}
We compute the formal geometric quantization of a symplectic toric manifold and then we count the Bohr-Sommerfeld leaves. We see that they are both the same and, in particular, they coincide with the count of the integer points in the image of the moment map (with sign) of the toric action.

In view of Theorem \ref{th:QR} (Quantization commutes with reduction), the formal geometric quantization of a symplectic toric manifold $(M^{2n},\omega,\,u)$ is given by
\begin{equation}
    Q(M) = \bigoplus_{\alpha\in\mathbb{Z}^n} Q(M//_\alpha T) \alpha.
\end{equation}
Notice that the sum is taken over all weights $\alpha$ of $T$.

Suppose $\mu:M\to\mathfrak{t}$ is the moment map of the toric action. Then, the reduced spaces $M//_\alpha T$ are either empty if $\alpha$ is not in $\mu(M)$ or a point if it is. Since the quantization of each single point is given by $\mathbb{C}$, we have that

\begin{equation}
    Q(M) = \bigoplus_{\alpha\in\mathbb{Z}^n\cap \mu(M)} \mathbb{C}(\alpha).
\end{equation}

Then, the formal geometric quantization of $M$ is given by as many copies of $\mathbb{C}$ as integer points in the image of the moment map.

On the other hand, by Theorem \ref{th:BSequalIntegerpoints}, the Bohr-Sommerfeld quantization is given by the count of Bohr-Sommerfeld leaves of $M$, which coincides with the integer points in the image of the moment map.
\end{proof}

Theorem \ref{thm:GQcoincidesFGQ} also holds when the manifold is $b$-symplectic.

\begin{theorem}
Let $(M^{2n},Z,\omega,\mu)$ be a $b$-symplectic toric manifold such that the image of the moment map is simply connected. Then, the formal geometric quantization of $M$ coincides with the Bohr-Sommerfeld quantization with sign.
\label{thm:b-GQcoincidesFGQ}
\end{theorem}

\begin{proof}
Recall from \cite{GuilleminMirandaWeitsman18b} that for any $b$-symplectic toric manifold $(M,Z,\omega,\mu)$, the quantization space $Q(M)$ is defined as the vector space such that the following equality holds
\begin{equation}
    (Q(M) \otimes Q(N))^{\alpha} = \varepsilon(\alpha)Q((M\times N)//_\alpha T)
    \label{eq:qmbeta}
\end{equation}
for any compact symplectic manifold $N$ and any weight $\alpha$ of $T$, where $T$ is the torus generating the action with moment map $\mu$ \cite{GuilleminMirandaWeitsman18}.

Since we are considering the case of a toric action, the reduced spaces are empty or just points. Observe that the right hand side of equation \ref{eq:qmbeta} is symplectic. Therefore we can apply the \textit{quantization commutes with reduction} scheme.

We can take $N$ to be the coadjoint orbit of this $T^n$-action of integral weight $\alpha$, which is in the integral lattice by definition. Applying equation \ref{eq:qmbeta} to equation \ref{eq:quantredsp} from Theorem \ref{th:QR}, we see that the formal geometric quantization of a $b$-symplectic manifold is given by the direct sum of the reduced spaces $Q(M//_\alpha T)$ at each $\alpha$.

The weights $\alpha$ are precisely the points in the integer lattice of the moment map of the toric action, meaning that $\{\alpha\}=\mathbb{Z}^n\cap \mu(M)$. On the other hand, Then, the quantization of $M$ is \begin{equation}
    Q(M) = \bigoplus_\alpha \varepsilon(\alpha) Q(M//_\alpha T) \alpha = \bigoplus_{\alpha\in\mathbb{Z}^n\cap \mu(M)} \varepsilon(\alpha) \mathbb{C}(\alpha).
    \label{eq:finalcountb1}
\end{equation}

On the other hand, by Corollary \ref{cor:b-BSequalIntegerpoints} the Bohr-Sommerfeld set of $M$ coincides with the lattice of integer points in the image of $\mu$. Therefore, by Definition \ref{def:geomquantwithsign3}, the Bohr-Sommerfeld quantization with sign of $M$ is
\begin{equation}
    \tilde{Q}(M)=\bigoplus_{b\in B_{BS}} \epsilon(b) \mathbb{C}(b) = \bigoplus_{b\in\mathbb{Z}^n\cap \mu(M)} \epsilon(b) \mathbb{C}(b).
    \label{eq:finalcountb2}
\end{equation}

Finally, for any point $p$ in the Bohr-Sommerfeld set $B_{BS}$, the sign $\epsilon(p)$ coincides with the sign $\varepsilon(p)$. By definition, both of them are $+1$ if the relative orientations of the symplectic forms on the component of $\mu^{-1}(p)$ of $M \setminus Z$ and the overall orientation of $M$ agree and $-1$ otherwise. Hence, $Q(M)=\tilde{Q}(M).$
\end{proof}

\bibliographystyle{alpha}
\bibliography{BSQuantofbSTM}

\begin{thebibliography}{DGMW95}

\bibitem[BLS21]{BravermanLoizidesSong}
Maxim Braverman, Yiannis Loizides, and Yanli Song.
\newblock Geometric quantization of b -symplectic manifolds.
\newblock {\em J. Symplectic Geom.}, 19(1):1--36, 2021.

\bibitem[Del88]{Delzant88}
Thomas Delzant.
\newblock Hamiltoniens p\'{e}riodiques et images convexes de l'application
  moment.
\newblock {\em Bull. Soc. Math. France}, 116(3):315--339, 1988.

\bibitem[DGMW95]{DuistermaatGuilleminMeinrenkenWu95}
Hans Duistermaat, Victor Guillemin, Eckhard Meinrenken, and Siye Wu.
\newblock Symplectic reduction and {R}iemann-{R}och for circle actions.
\newblock {\em Math. Res. Lett.}, 2(3):259--266, 1995.

\bibitem[Eli90]{Eliasson90}
L.~H. Eliasson.
\newblock Normal forms for {H}amiltonian systems with {P}oisson commuting
  integrals---elliptic case.
\newblock {\em Comment. Math. Helv.}, 65(1):4--35, 1990.

\bibitem[GMP11]{GuilleminMirandaPires11}
Victor Guillemin, Eva Miranda, and Ana~Rita Pires.
\newblock Codimension one symplectic foliations and regular {P}oisson
  structures.
\newblock {\em Bull. Braz. Math. Soc. (N.S.)}, 42(4):607--623, 2011.

\bibitem[GMP14]{GuilleminMirandaPires14}
Victor Guillemin, Eva Miranda, and Ana~Rita Pires.
\newblock Symplectic and {P}oisson geometry on {$b$}-manifolds.
\newblock {\em Adv. Math.}, 264:864--896, 2014.

\bibitem[GMPS15]{GuilleminMirandaPiresScott15}
Victor Guillemin, Eva Miranda, Ana~Rita Pires, and Geoffrey Scott.
\newblock Toric actions on {$b$}-symplectic manifolds.
\newblock {\em Int. Math. Res. Not. IMRN}, pages 5818--5848, 2015.

\bibitem[GMPS17]{GuilleminMirandaPiresScott17}
Victor Guillemin, Eva Miranda, Ana~Rita Pires, and Geoffrey Scott.
\newblock Convexity for {H}amiltonian torus actions on {$b$}-symplectic
  manifolds.
\newblock {\em Math. Res. Lett.}, 24(2):363--377, 2017.

\bibitem[GMW18a]{GuilleminMirandaWeitsman18b}
Victor~W. Guillemin, Eva Miranda, and Jonathan Weitsman.
\newblock Convexity of the moment map image for torus actions on
  {$b^m$}-symplectic manifolds.
\newblock {\em Philos. Trans. Roy. Soc. A}, 376(2131):20170420, 6, 2018.

\bibitem[GMW18b]{GuilleminMirandaWeitsman18}
Victor~W. Guillemin, Eva Miranda, and Jonathan Weitsman.
\newblock On geometric quantization of {$b$}-symplectic manifolds.
\newblock {\em Adv. Math.}, 331:941--951, 2018.

\bibitem[GMW21]{GuilleminMirandaWeitsman21}
Victor~W. Guillemin, Eva Miranda, and Jonathan Weitsman.
\newblock On geometric quantization of {$b^m$}-symplectic manifolds.
\newblock {\em Math. Z.}, 298(1-2):281--288, 2021.

\bibitem[GS82]{GuilleminSternberg82}
V.~Guillemin and S.~Sternberg.
\newblock Geometric quantization and multiplicities of group representations.
\newblock {\em Invent. Math.}, 67(3):515--538, 1982.

\bibitem[GS83]{GuilleminSternberg83}
V.~Guillemin and S.~Sternberg.
\newblock The {G}elfand-{C}etlin system and quantization of the complex flag
  manifolds.
\newblock {\em J. Funct. Anal.}, 52(1):106--128, 1983.

\bibitem[Ham08]{Hamilton08}
Mark~D. Hamilton.
\newblock The quantization of a toric manifold is given by the integer lattice
  points in the moment polytope.
\newblock In {\em Toric topology}, volume 460 of {\em Contemp. Math.}, pages
  131--140. Amer. Math. Soc., Providence, RI, 2008.

\bibitem[Ham10]{Hamilton10}
Mark~D. Hamilton.
\newblock Locally toric manifolds and singular {B}ohr-{S}ommerfeld leaves.
\newblock {\em Mem. Amer. Math. Soc.}, 207(971):vi+60, 2010.

\bibitem[HM16]{HochsMathai16}
Peter Hochs and Varghese Mathai.
\newblock Formal geometric quantisation for proper actions.
\newblock {\em J. Homotopy Relat. Struct.}, 11(3):409--424, 2016.

\bibitem[KMS16]{KiesenhoferMirandaScott16}
Anna Kiesenhofer, Eva Miranda, and Geoffrey Scott.
\newblock Action-angle variables and a {KAM} theorem for {$b$}-{P}oisson
  manifolds.
\newblock {\em J. Math. Pures Appl. (9)}, 105(1):66--85, 2016.

\bibitem[Kos70]{Kostant70}
Bertram Kostant.
\newblock Quantization and unitary representations. {I}. {P}requantization.
\newblock In {\em Lectures in modern analysis and applications, {III}}, pages
  87--208. Lecture Notes in Math., Vol. 170. Springer-Verlag, Berlin, 1970.

\bibitem[LLSS21]{LinLoizidesSjamaarSong21}
Yi~Lin, Yiannis Loizides, Reyer Sjamaar, and Yanli Song.
\newblock Log symplectic manifolds and $[q,r]=0$, 2021.

\bibitem[Mei96]{Meinrenken96}
Eckhard Meinrenken.
\newblock On {R}iemann-{R}och formulas for multiplicities.
\newblock {\em J. Amer. Math. Soc.}, 9(2):373--389, 1996.

\bibitem[MM21]{MirMiranda21}
Pau Mir and Eva Miranda.
\newblock Geometric quantization via cotangent models.
\newblock {\em Anal. Math. Phys.}, 11(3):Paper No. 118, 2021.

\bibitem[MO21]{MirandaOms}
Eva Miranda and C\'{e}dric Oms.
\newblock The singular {W}einstein conjecture.
\newblock {\em Adv. Math.}, 389:Paper No. 107925, 41, 2021.

\bibitem[Par09]{Paradan09}
Paul-\'{E}mile Paradan.
\newblock Formal geometric quantization.
\newblock {\em Ann. Inst. Fourier (Grenoble)}, 59(1):199--238, 2009.

\bibitem[\'{S}77]{Sniatycki77}
Jedrzej \'{S}niatycki.
\newblock On cohomology groups appearing in geometric quantization.
\newblock In {\em Differential geometrical methods in mathematical physics
  ({P}roc. {S}ympos., {U}niv. {B}onn, {B}onn, 1975)}, pages 46--66. Lecture
  Notes in Math., Vol. 570, 1977.

\bibitem[Wei01]{Weitsman01}
Jonathan Weitsman.
\newblock Non-abelian symplectic cuts and the geometric quantization of
  noncompact manifolds.
\newblock {\em Lett. Math. Phys.}, 56(1):31--40, 2001.
\newblock EuroConf\'{e}rence Mosh\'{e} Flato 2000, Part I (Dijon).

\bibitem[Wil36]{Williamson36}
John Williamson.
\newblock On the {A}lgebraic {P}roblem {C}oncerning the {N}ormal {F}orms of
  {L}inear {D}ynamical {S}ystems.
\newblock {\em Amer. J. Math.}, 58(1):141--163, 1936.

\end{thebibliography}

\end{document}